\documentclass{amsart}

\usepackage[T1]{fontenc}
\usepackage{lmodern}
\usepackage{amsmath,amssymb,amsthm,mathtools}
\usepackage{xcolor}
\usepackage{hyperref}
\usepackage[nameinlink,capitalize]{cleveref}

\hypersetup{
  pdftitle={A Kovari--Sos--Turan theorem for stable hypergraphs},
  pdfauthor={Aris Papadopoulos},
  colorlinks=true,
  linkcolor=blue!45!black,
  citecolor=red!45!black,
  urlcolor=blue!45!black
}

\newtheorem{theorem}{Theorem}
\newtheorem{theoremA}{Theorem}

\newtheorem*{theorem*}{Theorem}
\newtheorem{proposition}[theorem]{Proposition}
\newtheorem{lemma}[theorem]{Lemma}
\newtheorem{corollary}[theorem]{Corollary}

\theoremstyle{definition}

\theoremstyle{remark}

\newcommand{\Kpart}[2]{K^{(#1)}_{#2,\ldots,#2}}

\title[K\H{o}v\'ari--S\'os--Tur\'an theorem for stable hypergraphs]{A note on the K\H{o}v\'ari--S\'os--Tur\'an theorem for stable hypergraphs}
\author{Aris Papadopoulos}
\thanks{\hspace{-1.4em}\textbf{Funding.} During part of the writing of this paper, the author was supported by Julia Wolf's Open Fellowship from the UK Engineering and Physical Sciences Research Council (EP/Z53352X/1).}
\address{DPMMS, Centre for Mathematical Sciences, Wilberforce Road, Cambridge CB3 0WB, United Kingdom}
\email{adp70@cam.ac.uk}
\begin{document}

\begin{abstract}
  In this note, we prove a stronger version of the K\H{o}v\'ari--S\'os--Tur\'an theorem for partition-wise $k$-stable $r$-hypergraphs. More precisely, we show that for every $k,r\in\mathbb{N}_{\geq 2}$ there is $\eta=\eta(r,k)>0$ such that if $H=(V;E)$ is a partition-wise $k$-stable $r$-uniform $K^{(r)}_{t,\ldots,t}$-free hypergraph with $|V|=n$, then $|E| = O_{r,k,t}(n^{r-\eta})$. Crucially, $\eta$ is independent of $t$. The proof follows the pseudofinite regime used by Chernikov and Starchenko to prove an analogous version of Erd\H{o}s-Hajnal for stable hypergraphs in \cite{ChernikovStarchenko2018}.
\end{abstract}

\maketitle

\section{Introduction}
\noindent Let $r,t\geq\in\mathbb{N}_{\geq 2}$.  An \emph{$r$-uniform hypergraph}, henceforth $r$-graph, is a pair $H=(V;E)$ where $E\subseteq \binom{V}{r}$, i.e. $E$ is a symmetric and irreflexive $r$-ary relation. We say that $H$ is \emph{$\Kpart{r}{t}$-free} if whenever $A_i\subseteq V$ with $|A_i| = t$ (for $i=1,\dots,r$), then $A_1\times\cdots\times A_r\nsubseteq E$. Let $H=(V,E)$ be an $r$-graph. A celebrated result of Erd\H{o}s \cite{Erdos1964} generalising the well-known K\H{o}v\'ari--S\'os--Tur\'an theorem, from graphs (i.e. $2$-graphs) to $r$-graphs gives $|E|=O_{r,t}\left(n^{\,r-\frac{1}{t^{r-1}}}\right)$, when $|V_i|=n$ and $H$ is $\Kpart rt$-free. 
  
In this note we show that if $E$ is also assumed to be \emph{partition-wise $k$-stable} (see \Cref{sec:preliminaries}), for some $k\in\mathbb{N}$, then (the fraction in) the exponent of Ed\H{o}s's bound above does not depend on $t$. More precisely, our main result is the following.

\begin{theoremA}\label{thm:intro}
  For every $r,k\geq\in\mathbb{N}_{\geq 2}$ there is $\eta=\eta(r,k)>0$ so that: For all $t\in\mathbb{N}_{\geq_2}$ there is $C=C(r,k,t)>0$ such that, if $H=(V;E)$ is a partition-wise $k$-stable $r$-graph with $|V|=n$, and $H$ is $\Kpart rt$-free, then $|E|\leq Cn^{r-\eta}$.
\end{theoremA}

Our proof goes via Hrushovski's coarse pseudofinite dimension (which, for the sake of some self-containment, we recall in \cref{sec:pseudofinite}), and is very much qualitative, as it does not give an explicit value of $\eta(r,k)$. We remark here that since the edge relation of edge-stable graphs has finite VC-dimension, our result for $r=2$ is just a special case of \cite[Theorem~2.1]{semialgZar}. It is new however for $r\geq 3$. Let us also remark that it is very likely that quantitative proofs in the stable context may well exist, but our goal with this note was to treat the K\H{o}v\'ari--S\'os--Tur\'an theorem as a ``dimensional'' statement, so we purposefully did not go down the path of trying to count things more carefully. 

We are inspired partly by \cite[Proposition~3.1]{ChernikovStarchenko2018} and partly by \cite[Proposition~7.7]{EleftheriouPapadopoulos2025}. The former shows, using coarse pseudofinite dimension and local stability, that a partition-wise stable $r$-graph has a homogeneous set of positive dimension, while the latter shows that in a stable 1-based context, (under an additional non-algebraicity hypothesis) this homogeneous set must consist of hyperedges.  Unfortunately, neither one of these results can be applied directly here, but we remedy this with \cref{prop:box}, which strengthens the conclusion of \cite[Proposition~3.1]{ChernikovStarchenko2018} under a full dimensionality assumption (akin to the non-algebraicity hypothesis mentioned of \cite[Proposition~7.7]{EleftheriouPapadopoulos2025}), to construct a Cartesian product of positive dimension sets contained in the hyperedge relation. Having achieved this, the proof main theorem, is a relatively easy exercise (\emph{pose}) in pseudofinite yoga, given in full detail in \cref{sec:main-theorem}. 

\subsection*{Organisational matters} The rest of this note is organised as follows. In \cref{sec:preliminaries} we briefly recall the relevant tools from the local-stability and coarse pseudofinite dimension toolboxes and prove some (easy) preparatory results. In \cref{sec:full-box} we put these tools to work to prove our main technical proposition, discussed above, (\cref{prop:box}). Finally in \cref{sec:main-theorem} we spell out the  rest of the details of the proof of \cref{thm:intro}, which we deduce from the analogous partite statement, \cref{thm:main}.

\subsection*{Acknowledgments.}
I am extremely grateful to Amador Martin-Pizarro for the conversations, during the Leeds workshop on Tame Geometry and Combinatorics (March 2026), which inspired this note, and for suggesting the inclusion of \cref{cor:amador}. I am also very grateful to Rosario Mennuni and Mira Tartarotti for comments on earlier versions, and Ioannis Eleftheriadis for other comments. Finally, I want to thank Martin Bays and Atticus Stonestorm for helpful discussions.

\subsection*{AI Declaration.} The results in this paper are neither optimal, nor are they particularly difficult. They are however, in my humble opinion, pretty cool. I have little doubt that an LLM could get rather far into producing some version of my arguments, but where would the fun in that be... Anyway, no generative AIs were used during the preparation of this paper.

\section{Preliminaries}\label{sec:preliminaries}

\noindent We start by fixing some notation. For $r\in\mathbb{N}$ we write $[r] =\{1,\dots,r\}$. Given a cartesian product (sometimes called a box in the sequel) $V=\prod_{i=1}^r V_i$ of $r$-many sets and $j\in [r]$, we write $V_{\neq j} = \prod_{i\in[r]\setminus\{j\}}V_i$. Given sets $A$ and $B$, a binary relation $E\subseteq A\times B$ and $b\in B$ we write $E_b = \{a\in A:(a,b)\in E\}$, and more generally, if $B'\subseteq B$ write $E_{B'} =\{a\in A:\exists b\in B'\text{ s.t. }(a,b)\in E\}$, so that $E_{B'}=\bigcup_{b\in B} E_b$. Finally, for an $L$-formula $\phi(x;y)$ we write $\phi(x;y)^1$ for $\phi(x;y)$ and $\phi(x;y)^0$ for $\lnot\phi(x;y)$.

\subsection{Local stability} Let $k\in\mathbb{N}$. We say that a binary relation $E\subseteq V\times W$ has the \emph{$k$-order property} if there exist $v_1,\dots,v_k\in V$ and $w_1,\dots,w_k\in W$ such that for all $i,j\leq k$ we have that $(v_i,w_j)\in E$ if and only if $i\leq j$. We say that $E$ is \emph{$k$-stable} if it does not have the $k$-order property, and we say that it is \emph{stable} if it is $k$-stable, for some $k\in\mathbb{N}$.

 More generally, we call an $r$-ary relation $E\subseteq\prod_{i=1}^rV_i$ \emph{partition-wise ($k$-)stable} if for all $j\in[r]$, the relation $E$ viewed as a binary relation between $V_j$ and $V_{\neq j}$ is ($k$-)stable. We briefly remark here that partition-wise stability implies slice-wise stability, but not vice versa \cite{ChernikovTowsner}, and that Boolean combinations of partition-wise $k$-stable formulas (possibly with some variables replaced by parameters) are partition-wise $K$-stable, for some $K\in\mathbb{N}$, see \cite[Chapter~2]{Shelah1990}
 
The two main ingredients from local stability that we will use below are definability of ($\varphi$-)types and Shelah 2-rank. We will recall the latter and refer the reader to \cite[Lemma I.2.1]{Pillay} for a precise formulation of the former. 

Let $M\vDash T$ and $\varphi(x;y)$ a formula. Given $A\subseteq M^{x}$ and $r\in\mathbb{N}$ we define $R_\varphi(A)\geq r$ by induction, as follows: $R_\varphi(A)\geq 0$ if and only if $A$ is non-empty, and $R_\varphi(A)\geq r+1$ if and only if there is some $b\in M^{y}$ such that $R_\varphi(A\cap\phi(M;b)), R_\varphi(A\setminus\phi(M;b))\geq r$. For completeness, we set $R_\varphi(\emptyset)=-\infty$. Finally, for non-empty $A\subseteq M^x$ we set $R_\varphi(A) = \sup \{r\in\mathbb{N}:R_\varphi(A)\geq r\}$ if the supremum exists and $R_\varphi(A)=\infty$, otherwise. When $\varphi(x,y)$ is stable, $R_\varphi(A)\in\mathbb{N}$ for any $A\subseteq M^x$. For more precise statements see \cite[Chapter~II]{Shelah1990} or \cite[Section~3]{ChernikovNotes}.

\subsection{Coarse pseudofinite dimension}\label{sec:pseudofinite}
To quote Hrushovski, \cite{Hrushovski2013}: \emph{``The logarithm of the cardinality of a nonstandard finite set behaves like a dimension theory, as soon as one factors out a nontrivial convex subgroup of the nonstandard reals.'' }This wonderfully concise description of pseudofinite dimension contains essentially all the information that we recall in this subsection, lifted mainly from \cite[Section~2]{Hrushovski2013} and \cite{ChernikovNotes}, and included here mainly for completeness. 

For each $n\in\mathbb{N}$ let $M_n$ be a finite $L$-structure, with $|M_n|\to\infty$, as $n\to\infty$, and let $\mathcal{U}\in\beta\mathbb{N}\setminus\mathbb{N}$ be a non-principal ultrafilter. We write $M=\prod_{n\in\mathbb{N}}M_n/\mathcal{U}$ for the ultraproduct of the $M_i$, modulo $\mathcal{U}$. A subset $A\subseteq M^d$ is called \emph{internal} if $A = \prod_{n\in\mathbb{N}} A_i/\mathcal{U}$, where $A_n\subseteq M_n^d$, for all $n\in\mathbb{N}$. For an internal $A=\prod_{n\in\mathbb{N}}A_n/\mathcal{U}$, its \emph{nonstandard cardinality} is $|A|=\prod_{n\in\mathbb{N}}|A_n|/\mathcal{U}$.  

Let $V = \prod_{n\in\mathbb{N}} V_i/\mathcal{U}\subseteq M^d$ be a non-empty internal set and $X\subseteq V$ a non-empty internal subset. We define its \emph{coarse pseudofinite dimension}, $\delta(X)$, by taking the standard part of the logarithm of its cardinality, scaled by the logarithm of the cardinality of the ambient space, that is $\delta(X)=\mathsf{st}\!\left(\frac{\log|X|}{\log |V|}\right)$.

In \cref{lem:dimension-rules}, which is given below without proof, we record some basic properties of $\delta$ that will be used in the sequel, sometimes even with explicit mention. They are all either standard or trivial consequences of (nonstandard) cardinality and logarithms.

\begin{lemma}[Hrushovski, \cite{Hrushovski2013}]\label{lem:dimension-rules}
Let $M$ and $\delta$ be as above, and $X,Y\subseteq M$ be internal.
  \begin{enumerate}
  \item If $X\subseteq Y$, then $\delta(X)\leq\delta(Y)$. If, in addition, $|X|\geq |Y|/K$ for some (standard) $K \geq 1$, then $\delta(X)=\delta(Y)$.
  \item $\delta(X\cup Y)=\max\{\delta(X),\delta(Y)\}$. In particular, if $Y\subseteq X$ and $\delta(Y)<\delta(X)$, then $\delta(X\setminus Y)=\delta(X)$.
  \item If $X$ and $Y$ are nonempty, then $\delta(X\times Y)=\delta(X)+\delta(Y)$.
  \item If $0\leq\gamma<\delta(X)$, then there is an internal
    $X'\subseteq X$ with $\delta(X')=\gamma$.
  \end{enumerate}
\end{lemma}

\subsection{Set-up} Let us now discuss the the modest interplay between coarse pseudofinite dimension, local stability, and Shelah 2-rank needed for our purposes. \cref{lem:minimal-rank}, the small general result of the remainder of this section is just a rather elementary observation, and certainly well-known, but we give a proof, regardless. The goal of this subsection is to prove \cref{lem:lower-rank-fibers}, which will be the crucial ingredient in the induction step, in the proof of \cref{prop:box}.

Until the end of the section, the notation will be cumulative, and it will start with a fixed sequence of finite $L$-structures $(M_n:n\in\mathbb{N})$, with $|M_n|\to\infty$ as $n\to\infty$ and a fixed non-principal ultrafilter $\mathcal{U}\in\beta\mathbb{N}\setminus\mathbb{N}$. We let $M=\prod_{n\in\mathbb{N}}M_n/\mathcal{U}$. We also fix a stable formula $\varphi(x;y)$ and a non-empty internal subset $V\subseteq M^{x}$, with $\delta(V) = d$, for some (standard) $d\geq 0$. We let $r$ be the minimum value of $R_\varphi(A)$, as $A$ ranges over all internal subsets of $V$ with $\delta(A)=\delta(V)=d$. 

\begin{lemma}\label{lem:minimal-rank}
For each $A\subseteq V$ with $(\delta(A),R_\varphi(A)) = (d,r)$ there is a unique $\varphi$-type $p_A\in S^\varphi(M)$ such that for all $b\in M^{y}$ and all $\varepsilon\in\{0,1\}$ the following are equivalent:
\begin{enumerate}
  \item $\phi(x;b)^\varepsilon\in p_A$.
  \item $\delta(A\cap\phi(M;b)^\varepsilon) = d$ (and $\delta(A\cap\phi(M;b)^{1-\varepsilon})<d$).
  \item $R_\varphi(A\cap\phi(M;b)^\varepsilon)=r$ (and $R_\varphi(A\cap\phi(M;b)^{1-\varepsilon})<r$).
\end{enumerate}
Moreover, the set $D(p_A)= \{b\in M^{|y|}:\varphi(x;b)\in p_A\}$ is internal.
\end{lemma}

\begin{proof}
  Let $A\subseteq V$ be as in the statement of the lemma and suppose that $B\subseteq A$ has $\delta(B)=d$. Clearly $R_\varphi(B)\geq r$, since $r$ was chosen to be minimum, so $R_\varphi(B) = r$, since $B\subseteq A$. Given some $b\in M^y$, write $B_\varepsilon = A\cap\phi(M;b)^\varepsilon$, for $\varepsilon\in\{0,1\}$. Clearly, for at least one $\varepsilon\in \{0,1\}$ we have that $\delta(B_\varepsilon) = d$, and thus, by the above $R_\varphi(B_\varepsilon) = r$. If we also had $\delta(B_{1-\varepsilon}) = d$, then arguing as in the last sentence $R_\varphi(B_{1-\varepsilon}) = r$, but then $R_\varphi(A) \geq r+1$, hence $\delta(B_{1-\varepsilon}) <d$. In fact, this shows that $R_\varphi(B_{1-\varepsilon})<r$, in general. Let $p_A = \{\phi(x;b)^{\varepsilon(b)}:b\in M^y\}$. That $p_A$ is unique, maximal, and satisfies $(1)\Leftrightarrow(2)\Leftrightarrow(3)$ is by construction, so it remains to show that it is finitely consistent. Let $\varphi(x,b_1)^{\varepsilon(b_1)},\dots,\varphi(x,b_n)^{\varepsilon(b_n)}\in p_A$. Then, $\delta(A\cap\varphi(x,b_i)^{1-\varepsilon(b_i)})<d$ for each $i\in[n]$. Write $C=A\cap \left(\bigcap_{i=1}^n\varphi(x;b_i)^{\varepsilon(b_i)}\right)$. A straightforward dimension argument
   then shows that $\delta(C) = d\geq 0$, and thus $p_A$ is consistent (and actually finitely satisfiable in $A$). The moreover part follows from definability of $\varphi$-types, for stable $\varphi$, and the fact that $M$-definable sets are internal.
\end{proof}

Continuing with the notation already established, fix some non-empty internal set $B$ and some non-empty internal $E\subseteq A\times B$ such that for all $b\in B$ we have that $R_\varphi(E_b)<r$. Let $b_\star\in B$ be so that for all $b\in B$ we have that $|E_b|\leq |E_{b_\star}|$, and note that by our rank assumption and \cref{lem:minimal-rank} we have that $\delta(E_{b_\star})<d$. 

\begin{corollary}\label{lem:lower-rank-fibers}
  In the notation above, for any non-empty internal $B'\subseteq B$ we have $\delta(E_{B'})\leq \delta(E\cap (A\times B'))\leq \delta(B')+\delta(E_{b_\star})$. In particular, if $\delta(B')+\delta(E_{b_\star})<d$, then $\delta(A\setminus E_{B'})=d$.
\end{corollary}
\begin{proof}
  Let $B' = \prod_{n\in\mathbb{N}} B_n'/\mathcal{U}$, with $\emptyset\neq B'_n\subseteq B_n$ for all $n\in\mathbb{N}$. By our choice of $b_\star$, for all $n\in\mathbb{N}$ we have $|E_n\cap(A_n\times B_n')| = \sum_{b\in B_n'}|(E_n)_b|\leq |B_n'||(E_n)_{b_\star}|$. Thus the claimed inequality follows by taking logarithms. The ``in particular'' part is immediate. 
\end{proof}

\section{The full-dimension box lemma}\label{sec:full-box}
The main technical tool of this paper is \cref{prop:box}, which is proved in this section. The proof is by induction on the arity, and will at the crucial point construct a set satisfying the assumptions discussed just before \cref{lem:lower-rank-fibers}. Namely, we will build an internal set $D$ of full dimension, one arity below, use induction to extract a box, and  show that by construction all fibres of the complement of the hyperedge relation that go through our box have smaller rank. This will guarantee, by the aforementioned corollary that the appropriate complement has positive dimension, and allow the induction to go through. The proof will hopefully be more enlightening.

\begin{proposition}\label{prop:box}
  Let $M$ be as above. Let $V_1,\ldots,V_r\subseteq M$ be nonempty internal subsets with $\delta(V_i)=d_i>0$, for each $i=1,\dots,r$. Let $E\subseteq V_1\times\cdots\times V_r$ be internal and partition-wise stable. If $\delta(E)=\sum_{i=1}^rd_i$, then there are internal $B_i\subseteq V_i$ such that $\delta(B_i)>0$ for each $i=1,\dots,r$ and $B_1\times\cdots\times B_r\subseteq E$.
\end{proposition}

\begin{proof}
  The proof is by induction on $r$, the base case being trivial. So, suppose $r>1$, and let $V_1,\dots,V_r$ and $E$ be as above. To facilitate notation, let $X=V_1$, $Y=V_2\times\cdots\times V_r$ and $\alpha=\delta(X)$, $\beta=\delta(Y) = \sum_{i=2}^rd_i$. 
  
  Take $X=\prod_{n\in\mathbb{N}}X_n/\mathcal{U}$, $Y=\prod_{n\mathbb{N}}Y_n/\mathcal{U}$, and $E=\prod_{n\in\mathbb{N}}E_n$, where, without loss of generality, $X_n,Y_n,E_n$ are nonempty and $E_n\subseteq X_n\times Y_n$ for each $n\in\mathbb{N}$. For $v\in X_n$, set $e_n(v)=|\{w\in Y_n:(v,w)\in E_n\}|$ and $X_n'=\left\{v\in X_n:
      e_n(v)\geq\frac{|E_n|}{2|X_n|}\right\}$. Take $X'=\prod_{n\in\mathbb{N}}X'_m/\mathcal{U}$. \\\vspace{-0.8em}
      
\noindent\textit{Claim 1.} $\delta(X')=\alpha$. \\\vspace{-0.8em}

\noindent \textit{Proof of Claim 1.} By an elementary counting argument, for all $n\in\mathbb{N}$, we have that $|E_n\setminus(X'_n\times Y_n)| =\sum_{v\in X_n\setminus X'_n}e_n(x)\leq|X_n|\frac{|E_n|}{2|X_n|}=\frac{|E_n|}{2}$ and, in particular, $|E_n\cap(X'_n\times Y_n)|\geq\frac{|E_n|}{2}$. So, we have $\delta(E\cap(X'\times Y))=\delta(E) = \alpha+\beta$.  Since this relation is contained in $X'\times Y$, we have $\alpha+\beta\leq\delta(X')+\beta$, so $\alpha\leq \delta(X')$. On the other hand, $\delta(X')\leq\alpha$, since $X'\subseteq X$. \hfill$\qed_{\textit{Claim 1}}$ \\\vspace{-0.8em}

\noindent More generally, and by a similar counting argument we have:\vspace{-0.45em}
\begin{center}
  $(\star)$  For all internal $A\subseteq X'$ with $\delta(A)=\alpha$ we have  $\delta(E\cap(A\times Y)) = \alpha+\beta$.\vspace{-0.45em}
\end{center}
Indeed, given $A\subseteq X'$ as in $(\star)$, for all $n\in \mathbb{N}$ we have $|E_n\cap(A_n\times Y_n)| =\sum_{x\in A_n}e_n(x) \geq |A_n|\frac{|E_n|}{2|X_n|}$. Thus, $\delta\bigl(E\cap(A\times Y)\bigr) \geq\delta(A)+\delta(E)-\delta(X)=\alpha+\beta$, and the other inequality is trivial. 
  

 Since $E$ is internal, without loss of generality, after expanding the language we may assume that it is definable, say by $\varphi(x;y)$. Choose an internal subset $A=\prod_{n\in\mathbb{N}}A_n/\mathcal{U}\subseteq X'$ such that $\delta(A)=\alpha$ and $R_\varphi(A)$ is minimal amongst the internal subsets of $X'$ of dimension $\alpha$. 
%
By $(\star)$ we have $\delta(E\cap (A\times Y)) = \alpha+\beta$. In particular, $A$ meets the assumptions of \Cref{lem:minimal-rank}.
 
Our goal now is to find some partition-wise stable internal subset $D\subseteq Y$ with $\delta(D) = \beta$, so that our inductive hypothesis can be applied. To this end, by applying \cref{lem:minimal-rank} to $A$, we obtain a complete $\varphi$-type $p_A$ and an internal set $D(p_A)$, satisfying the conclusion of that lemma. Set $D=D(p_A)\cap Y$. Observe that $D$ is internal and stable since by definability of $\varphi$-types, the formula defining $D(p_A)$ is a Boolean combination of instances of a partition-wise stable formula, and thus remains partition-wise stable. \\\vspace{-0.8em}
  
\noindent\textit{Claim 2.}  $\delta(D)=\beta$ \\\vspace{-0.8em}

\noindent\textit{Proof of Claim 2.} Suppose towards contradiction that $\delta(D)<\beta$, so $\delta(Y\setminus D)=\beta$. To fix ideas, write $E_0=E\cap(A\times D)$ and $E_1=E\cap\bigl(A\times(Y\setminus D)\bigr)$, so that $E\cap(A\times Y)=E_0\sqcup E_1$. Clearly, $\delta(E_0) \leq\delta(A\times D)=\alpha+\delta(D)<\alpha+\beta$. Since $\delta\bigl(E\cap(A\times Y)\bigr)=\alpha+\beta$, it follows that $\delta(E_1)=\alpha+\beta$. By construction, for each $w\in Y\setminus D$, the fibre $(E_1)_w=A\cap E_w\notin p_A$, and thus $R((E_1)_w)<R_\varphi(A)$. Let $w_\star\in Y\setminus D$ be such that for all $w\in Y\setminus D$ we have that $|(E_1)_w|\leq |(E_1)_{w_\star}|$ and let $q=\delta((E_1)_{w_\star})$, and note that $q<\alpha$, by the discussion before \cref{lem:lower-rank-fibers}. Now, by \cref{lem:lower-rank-fibers} (applied to $Y\setminus D$) we have that $\delta(E_1)\leq\delta(Y\setminus D)+q =\beta+q<\alpha+\beta$, a contradiction. \hfill$\qed_{\textit{Claim 2}}$ \\\vspace{-0.8em}

Thus, we may apply our induction hypothesis to $D$, to obtain internal sets $B_i\subseteq V_i$ for $i=2,\dots,r$ with $\delta(B_i)> 0$ and such that $B = \prod_{i=2}^r B_i \subseteq D$. Let $E' = \{(a,b)\in A\times B: (a,b)\notin E\}$. By definition, for each $b\in B$ we have that $E'_b = A\setminus E_b$. \\\vspace{-0.8em}

\noindent\textit{Claim 3.}  For all $b\in B$ we have that $R_\varphi(E_b')<R_\varphi(A)$ \\\vspace{-0.8em}

\noindent\textit{Proof of Claim 3.} Immediate, after unpacking the definitions.
 \hfill$\qed_{\textit{Claim 3}}$ \\\vspace{-0.8em}

If $E'=\emptyset$, then $A\times B \subseteq E$, and since $\delta(A)>0$ we are done. Thus, we may assume that $E'\neq \emptyset$. Let $b_\star\in B$ be such that for all $b\in B$ we have that $|E'_b|\leq |E'_{b_\star}|$. By Claim 3, if we set $q = \delta(E'_{b_\star})$, then we have that $q<\alpha$ (see the discussion above \cref{lem:lower-rank-fibers}).

For $i=2,\dots,r$, let $\gamma_i=\min\!\left\{\frac{\delta(B_i)}{2},\frac{\alpha-q}{2(r-1)}\right\}$, so that $0<\gamma_i<\delta(B_i)$ for each $i=2,\dots,r$, and $\sum_{i=2}^r\gamma_i<\alpha-q$.  By \cref{lem:dimension-rules}(4) for each $i=2,\dots,r$ we can find internal $B_i'\subseteq B_i$ such that $\delta(B_i') = \gamma_i$. Letting $B'=\prod_{i=2}^rB'_i$ we have $\delta(B')=\sum_{i=2}^r\gamma_i<\alpha-q$. So, by the ``in particular'' part of \cref{lem:lower-rank-fibers}, applied to $B'$, we have that $\delta(A\setminus\bigcup_{b\in B'}E_b')=\alpha>0$ (justifying the choice of the $\gamma_i$). Thus, taking $B_1=A\setminus\bigcup_{b\in B'}E_b'$, and $B_i = B_i'$ for $i=2,\dots,r$ gives the result.
\end{proof}

We can immediately deduce the partite analogue of \cite[Proposition~3.1]{ChernikovStarchenko2018}, which as remarked in the introduction does not immediately imply \Cref{prop:box}.

\begin{corollary}\label{cor:amador}
 Let $M,V_1,\dots,V_r$ be as in \cref{prop:box}, and let $E\subseteq V_1\times\cdots\times V_r$ be internal partition-wise stable. Then, there exist $B_i\subseteq V_i$ with $\delta(B_i)>0$, for $i\in[r]$ such that $B_1\times\cdots\times B_r\subseteq E$ or $B_1\times\cdots\times B_r\cap E=\emptyset$
\end{corollary}
\begin{proof}
  Either $\delta(E) = \sum_{i=1}^r d_i$ or $\delta(\lnot E) = \sum_{i=1}^r d_i$, and in either case the result follows from applying \Cref{prop:box}.
\end{proof}

\section{Proof of the main theorem}\label{sec:main-theorem}

\noindent Our main theorem is an immediate corollary of the following result (which is the partition-wise stable hypergraph version of \emph{Zarankiewicz's problem}). Recall that an $r$-graph $H=(V,E)$ is called \emph{$r$-partite} if there is a partition $V=\bigsqcup_{i=1}^r V_i$ such that for all $i\in[r]$ and all $e\in E$ we have that $|e\cap V_i|=1$. 

\begin{theorem}\label{thm:main}
  For every $r,k\geq\in\mathbb{N}_{\geq 2}$ there is $\eta=\eta(r,k)>0$ such that for all $t\in\mathbb{N}_{\geq_2}$ there is $C=C(r,k,t)>0$ such that, if $H=(V_1,\ldots,V_r;E)$ is a partition-wise $k$-stable $r$-partite $r$-graph with $|V_1|=\cdots=|V_r|=n$, and $H$ is $\Kpart rt$-free, then $|E|\leq Cn^{r-\eta}$.
\end{theorem}

As commented in the introduction, the proof of \Cref{thm:main}, is just pseudofinite yoga. We assume that the result fails, and build a sequence of counterexamples of increasing sizes, with enough hyperedges to guarantee that, in the ultraproduct, the hyperedge has full dimension. Then, \cref{prop:box} guarantees us a positive dimensional box in the hyperedge relation, which is enough to obtain a contradiction.

\begin{proof}
  Suppose towards contradiction that the result is false. Then, there are $r,k\in\mathbb{N}_{\geq 2}$ such that for all $m\in\mathbb{N}_{\geq 2}$, there is some $t_m\geq 2$ such that all $C\geq 0$, and, in particular, for $C_m = m+t_m$ there is some $N_m\in\mathbb{N}$ and a partition-wise $k$-stable $r$-graph $H_m=(V_1^m,\dots,V_r^m;E_m)$ with $|V_i^m|=N_m$ is $\Kpart r {t_m}$-free but $|E_m|>C_m N_m^{\, r-\frac{1}{m}}$. It is easy to see that $N_m\to\infty$ as $m\to\infty$. Indeed, trivially, $|E_m|\leq N_m^r$ and thus $C_m=m+t_m<N_m^{\, \frac{1}{m}}$. Note also (for future use) that $t_m<N_m^{\,\frac{1}{m}}$.

  Now, let $\mathcal{U}\in\beta\mathbb{N}\setminus\mathbb{N}$, and let $H =(V_1,\dots,V_r;E)= \prod_{m\in\mathbb{N}} H_m/\mathcal{U}$. By \L o\'s, $E$ is partition-wise stable. Clearly, $\delta(V_i)=1$ for $i=1,\dots,r$, so $\delta(V_1\times\cdots\times V_r)=r$. By construction, for all $m\in\mathbb{N}$, we have that $\frac{\log|E_m|}{\log N_m}>r-\frac{1}{m}$, and since $\frac{1}{m}\to 0$ as $m\to \infty$, we also have that $\delta(E) = r$.
  
  By \cref{prop:box}, there are internal
  $B_i=\prod_{n\in\mathbb{N}}B_{i}^n/\mathcal{U}$ with $B_1\times\cdots\times B_r\subseteq E$ and $\delta(B_i)>0$, for $i=1,\dots,r$. Let $d = \min\{\delta(B_i):i=1,\dots,r\}>0$. We are now ready to produce our finite counterexample. Let $J_0=\{m\in\mathbb{N}:B_1^m\times\cdots\times B_{r}^m\subseteq E_m\}\in\mathcal{U}$ (since the $B_i$ are internal). Since $\delta(B_i)\geq d$, we have that $J_i = \bigl\{m\in\mathbb{N}:|B_i^m|\geq N_m^{\, \frac{d}{2}}\bigr\}\in\mathcal U$. Since $\mathcal{U}$ is non-principal, we have that $J_{-1}=\{m\in\mathbb{N}:m>\frac{2}{d}\}\in\mathcal{U}$. All in all, $J=\bigcap_{i=-1}^rJ_i\in\mathcal{U}$ and is thus nonempty. 
  
  Let $m\in J$, and consider $H_m = (V_1^m,\dots,V_r^m;E_m)$. By construction, we have that $B_1^m\times\cdots\times B_r^m\subseteq E_m$ and $|B_i^m|\geq N_m^{\,\frac{d}{2}}> N_m^{\, \frac{1}{m}}$. But, as we noted (for future use) earlier $t_m<N_m^{\,\frac{1}{m}}$ and thus $|B_i^m|\geq t_m$, for each $i=1,\dots,r$. But, then, since $B_{1}^m\times\cdots\times B_{r}^m\subseteq E_m$, $H_m$ is not $\Kpart r{t_m}$-free, a contradiction.
\end{proof}

We can finally conclude with the proof of \cref{thm:intro}.

\begin{proof}[Proof of \cref{thm:intro}]
  Given an $r$-graph $(V,E)$ we construct an $r$-partite $r$-graph $(V_1,\dots,V_r;E')$ as follows. Set $V_i=V\times\{i\}$, for $i\in[r]$, and $((v_1,1),\ldots,(v_r,r)\bigr)\in E'$ if and only if $(v_1,\ldots,v_r)\in E$. Since $E$ is symmetric and partition-wise $k$-stable, so is $E'$. It is also clear that if $E$ is $\Kpart rt$-free, then so is $E'$.  Indeed, suppose that $B_1'\times\cdots\times B_r'\subseteq E'$, with $|B_i'|=t$, and let $B_i = \{v\in V:v_i\in B_i'\}$. Since $E$ is irreflexive, $B_i\cap B_j=\emptyset$ whenever $i\neq j$, and hence $B_1\times\cdots\times B_r\subseteq E$ is a copy of $\Kpart r t$, a contradiction. Finally, since $|E|=|E'|$, applying \cref{thm:main} to $(V_1,\dots,V_r;E')$ finishes the proof.
\end{proof}

\end{document}